\documentclass{article}
\usepackage{latexsym}
\usepackage{graphics}
\usepackage{amsthm}
\newtheorem{theorem}{Theorem}[section]
\newtheorem{corollary}[theorem]{Corollary}
\newtheorem{lemma}[theorem]{Lemma}
\newtheorem{definition}[theorem]{Definition}
\newtheorem{remark}[theorem]{Remark}

\usepackage{amsfonts}
\newcommand{\Q}[0]{\mathbb{Q}}
\newcommand{\Z}[0]{\mathbb{Z}}
\renewcommand{\P}[0]{\mathbb{P}}

\newcommand{\mod}[0]{\mbox{ mod }}

\newcommand{\Rmnum}[1]{\expandafter\@slowromancap\romannumeral #1@}
\newcommand{\ord}[0]{\mbox{ord}}

\renewcommand{\v}[0]{\tilde{v}}
\newcommand{\w}[0]{\tilde{w}}
\renewcommand{\u}[0]{\tilde{u}}

\newcommand{\fr}[1]{\{#1\}}
\begin{document}
\title{Picard numbers of complex Delsarte surfaces with only isolated ADE-singularities.}
\author{Bas Heijne\footnote{Institut for Algebraic Geometry, Leibniz Universit\"at Hannover, Welfengarten 1, 30167 Hannover, Germany}}

\maketitle

\begin{abstract}
We give a classification of all Delsarte surfaces with only ADE-singularities. 
This results in several families.
For each of these families we give a closed fomula for the Picard number depending only on the degree.

\end{abstract}
\thanks{Funding by ERC StG 279723 (SURFARI) is gratefully acknowledged.
The author would like to thank Matthias Sch{\"u}tt, Tetsuji Shioda and Jaap Top for several helpfull suggestions and corrections.
}

\section{Introduction}

In \cite{ShiodaPic} Delsarte surfaces were first introduced and an algorithm to compute the Picard number of these surfaces was given.
This algorithm works for arbitrary characteristic.
In this text we will only consider the case where the surfaces are defined over the complex numbers.
Computation of the Picard number is in general a hard problem.
In \cite{Schuett}, the algorithm of \cite{ShiodaPic}  is one of the methods used to find Picard numbers of quintic surfaces with only ADE singularities.
In the current paper we extend these results to Delsarte surfaces of any degree with only isolated ADE singularities.
For non-singular Delsarte surfaces this has already been done in \cite{Jun}.

The main result of this paper is:
\begin{theorem}
For any degree $n\geq 6$, there are
 up to isomorphism at most $83$ Delsarte surfaces of degree $n\geq 6$ with only isolated ADE singularities.
 The possible cases, and for every case a closed formula for the Picard number, 
 is given in appendix \ref{tabel}.
\end{theorem}

Using the list in appendix \ref{tabel}, and considering the cases with $n=5$ as well, a short search gives the following.

\begin{corollary}
There are precisely three maximal (see \ref{definitions}) Delsarte surfaces with degree $n\geq 5$.
These are given by:
\[X^3YZ+Y^3ZU+XZ^3U+XYU^3=0,\]
\[X^5Y+XY^5+Z^5U+ZU^5=0 \mbox{ and }\] 
\[X^6+Y^6+Z^6+U^6=0.\] 

\end{corollary}

\section{Some general theory about Delsarte surfaces.} \label{definitions}
In this section we explain some of the theory on Delsarte surfaces that we will use.
We use the following definition of Delsarte surface.

\begin{definition}
 A Delsarte surface is a two dimensional subvariety in $\P^3$, which is defined as the zero set of a polynomial consisting of the sum of four monomials,
 such that the exponent matrix is invertible.
\end{definition}
Note that there always is a scaling of coordinates so that the constants in the polynomial are all one.
Hence we will assume we are in this situation.

In \cite{ShiodaPic} an algorithm is presented to compute the Picard number of these surfaces.
For later use we state an adapted version of this algorithm as given in \cite{HeijnePhD}.

Let $S$ be a projective Delsarte surface defined by the homogeneous polynomial
\[F=\sum_{i=1}^4 X^{a_{i1}}Y^{a_{i2}}Z^{a_{i3}}U^{a_{i4}}.\]
We construct the exponent matrix:
\[A=\left(
\begin{array}{cccc}
a_{11} &a_{12} &a_{13} &a_{14}\\
a_{21} &a_{22} &a_{23} &a_{24}\\
a_{31} &a_{32} &a_{33} &a_{34}\\
a_{41} &a_{42} &a_{43} &a_{44}\\
\end{array}
\right).\]

We define three row vectors $e_1=(1,0,0,-1)$, $e_2=(0,1,0,-1)$ and $e_3=(0,0,1,-1)$. 
From these we construct three vectors $\v=e_1A^{-1}$, $\w=e_2A^{-1}$ and $\u=e_3A^{-1}$.
Let $V$ be the $\Z$-module given by $V=\{(a_1,a_2,a_3,a_4)\in (\Q/\Z)^4 : a_1+a_2+a_3+a_4\equiv0(1).\}$
We will view $\v$, $\w$ and $\u$ as elements of $V$.
These vectors generate a finite $\Z$-module
\[L:=\{i\v+j\w+k\u\in (\Q/\Z)^4: i,j,k \in\Z\}.\]
We construct the set $L_0$ as the subset of $L$ where at least one of the coordinates is zero and $L_1$ as its complement.
\[L_0:=\{v=(v_1,v_2,v_3,v_4)\in L: \exists i : v_i= 0\}.\]
\[L_1:=\{v=(v_1,v_2,v_3,v_4)\in L: \forall i : v_i\neq 0\}.\]
We now define $\Lambda \subseteq L_1$ defined by the following property.
An element  $v=(v_1,v_2,v_3,v_4)\in L_1$ is an element of $\Lambda$ precisely when there exists a $t\in \Z$
such that for all $i$: $\ord_+(tv_i)=\ord_+(v_i)$ and  $\sum \fr{tv_i}\neq2$.
Here $\fr{\cdot}$ is the natural bijection between the set $\Q/\Z$ and the interval $[0,1)\cap \Q$ and $\ord_+(\cdot)$ stands for 
the order in the additive group $\Q/\Z$.

The Lefschetz number $\lambda$ is defined as the difference between the second Betti number of $S$ and its Picard number:
 $\rho=b_2-\lambda$.
The main result of \cite{ShiodaPic} is that the Lefschetz number can be computed by $\lambda=\#\Lambda$.
The number $b_2$ for a surface in $\P^3$ of degree $n$ with only ADE singularities is given by $b_2=n^3-4n^2+6n-2$.
The Hodge number $h^{1,1}=(2n^3+7)/3-2n^2$ is an upper bound for the Picard number: $\rho\leq h^{1,1}$
We call a surface maximal when we have equality in the last equation.
See \cite{lines} for more details.

\begin{remark}
It is important to note that if two surfaces are birationally equivalent, then they have the same Lefschetz number.
In particular we can compute the Lefschetz number on a singular (Delsarte) surface and use it to compute the Picard number of the desingularization.
\end{remark}

The set $L_1\setminus \Lambda$ consists of elements of a specific form.
This has been shown originally in \cite{ShiodaFer} for Fermat surfaces, but the results extend trivially to Delsarte surfaces.
There are three different types of elements $v\in L_1\setminus \Lambda$.
\begin{itemize}
 \item $v$ is such that $v_1+v_j\equiv 0(1)$ for some $j\in\{2,3,4\}$. In this case $v$ is called decomposable.
      The set of decomposable elements will be $D$.
\item $v$ is a permutation of a vector of the form $(a,1/2, 1/2+a, -2a)$ or $(a,1/2+a,1/2+2a,-4a)$ or $(a,1/3+a,2/3+a,-3a)$.
Here $a$ should be such that none of the coordinates is zero and the vector is not decomposable.
These $v$ are called regularly indecomposable. The set of indecomposable elements will be $R$. 
 \item $v$ is an exceptional element. It was proven in \cite{Aoki} and correctly formulated in \cite{AokiEr} that there are only a finite number of execeptional cases for a Fermat surface.
One can easily see that these are all the possible cases for a Delsarte surface as well.
The set of exceptional elements is explicitly known and has cardinality 22080.
  We will refer to these cases as irregular indecomposable and denote them by $I$.
\end{itemize}

\section{83 surfaces}

In this paper we consider Delsarte surfaces with only isolated ADE singularities.
There are several equivalent definitions of ADE singularities, see \cite{Durfee}.

\begin{definition}
 A singular point $P$ on a surface $S$ is ADE if it is locally isomorphic to one of the following types of singularities:
\begin{itemize}
 \item $X^{n+1}+Y^2+Z^2$ this is a singularity of type $A_n$, with $n\geq 1$.
 \item $X^{n-1}+XY^2+Z^2$ this is a singularity of type $D_n$, with $n\geq 4$.
 \item $X^3+Y^4+Z^2$ this is a singularity of type $E_6$.
 \item $X^3+XY^3+Z^2$ this is a singularity of type $E_7$.
 \item $X^3+Y^5+Z^2$ this is a singularity of type $E_8$.
\end{itemize}

\end{definition}

In \cite{reid} a method is given to effectivelly determine wether a singular point is ADE or not.

\begin{theorem}
There are up to isomorphism at most $83$ Delsarte surfaces of degree $n\geq 6$ with only ADE singularities.
\end{theorem}

\begin{proof}
We will first bound the number of Delsarte surfaces of a fixed degree $n\geq 6$ with only ADE-singularities.

Let $S$ be a Delsarte surface of degree $n$. Consider the point $P_x=(1:0:0:0)$.
Unless one of the monomials equals $X^n$, the point $P_x$ lies on $S$.
A necessary (but not sufficient) condition for $P_x$ to be either smooth or of ADE type is that one of the monomials defining $S$ is divisible by $X^{n-2}$. 

A similar argument applies for the points $(0:1:0:0)$,  $(0:0:1:0)$, $(0:0:0:1)$.
From this we see that if $S$ has only ADE singularities it must defined by a polynomial of the form.
\begin{equation}\label{vergelijking}
 F=X^{n-2}M_x+Y^{n-2}M_y+Z^{n-2}M_z+U^{n-2}M_u.
\end{equation}
Here $M_x$, $M_y$, $M_z$ and $M_u$ are degree 2 homogeneous monomials.
This already implies that there are at most $10^4$ possibilities for the equation $F$.
We will now in various staps reduce this number to $83$.

Now since $n\geq 6$ we see that the point $(1:0:0:0)$ is not of ADE type if 
$M_x$ is one of $Y^2$, $Z^2$ or $U^2$.
A similar condition applies to $M_y$, $M_z$ and $M_u$.

For fixed $n$ there are $7^4=2401$ degree $n$ polynomials for which this does not occur.
Hence we have at most $2401$ degree $n$ Delsarte surfaces with only ADE-singularities.

Using a computer algebra package we remove all $F$'s that are divisible by a coordinate,
 as well as duplicate results after permutations of coordinates.
This leaves us with $90$ surfaces.

For these $90$ surfaces we still have to check whether all the singular points are of ADE-type.
The following results will be helpful:

\begin{lemma}
Let $S$ be a Delsarte surface given by:
\[F=M_1+M_2+M_3+M_4.\] 
Here $M_i$ are monomials in $X$, $Y$, $Z$ and $U$.
Let $P$ be a singular point on $S$.
Then
\[M_i(P)=0,  \mbox{ for all } i.\]
\end{lemma}

\begin{proof}
Let $A$ be the exponent matrix of $S$.
By considering the partial derivatives of $F$ we find for a singular point $P$ that:
\[A \cdot 
\left(
\begin{array}{c}
M_1\\
M_2\\
M_3\\
M_4
\end{array}
\right)(P)
=
\left(
\begin{array}{c}
0\\
0\\
0\\
0
\end{array}
\right).
\]

Since $A$ is nonsingular this means that we have $M_i=0$ for all $i$.
\end{proof}

\begin{corollary}
Let $S$ be the Delsarte surface given by (\ref{vergelijking}).
Let $P$ be a singular point, then it has at least two zero coordinates.
\end{corollary}

\begin{proof}
None of the monomials contains all four coordinates, hence at least two of the coordinates have to vanish in order for all the $M_i$ to be zero.
\end{proof}

We first consider the singular points with three coordinates zero.
Without loss of generality we can assume that the singular point is the point $(0:0:0:1)$.

If $M_u=U^2$, then the point $(0:0:0:1)$ does not lie on the surface.
If $M_u=XU,YU,ZU$ then this point is non-singular.
Consider the case that $M_u$ is one of $XY$, $XZ$ or $YZ$. 
Without loss of generality we can then assume that $M_u=XY$.
The local equation of the point will then be:
\[xy+z^{n-2}M_z+\mbox{ terms of higher order.}\]
A necessary and sufficient condition for the point $(0:0:0:1)$ to be of ADE type is that $M_z\neq XY$.
For 7 of the 90 surfaces we find int this way a singular point not of ADE-type.

The remaining possible singular points are permutations of the point $(\eta:0:0:1)$, with $\eta$ nonzero.
We will assume that the singular point is the point $P=(\eta:0:0:1)$ and consider what happens depending on $M_x$.

\begin{itemize}
 
\item Assume that $M_x\in\{X^2,XU\}$.
For $P$ to be a point on the surface this means that $M_u\in\{U^2,XU\}$.
In this case $\eta$ has to be a specific root of unity.
It turns out that the partial derivatives with respect to $X$ and $U$ are then nonzero.
So in this case, the point $P$ is non-singular.

\item Assume that $M_x\in \{XY,UY\}$.
Then by considering the partial derivative with respect to $Y$ we find that $M_u\in \{XY,UY\}$.
Consider the affine chart with $U=1$ and transform the singular point to the origin.
Then we find a local equation of the form:
\[(\eta^a+\eta^b)y+ (a\eta^{a-1}+b\eta^{b-1})xy + \mbox{ higher order terms},\]
with $a\neq b$. This point is non-singular, or (if $\eta^a+\eta^b=0$) of ADE-type.

\item The case $M_x\in \{XZ,UZ\}$ is symmetric to the previous case.

\item The final case is $M_x=YZ$. 
This does not occur in our list of $83$ surfaces.
\end{itemize}
\end{proof}

The explicit equations of the $83$ surfaces are given in the appendix.

\begin{remark}
The table in appenindix \ref{tabel} was computed using the condition $n\geq 6$.
Also for $n=5$ the entries in the table are Delsarte surfaces with at most ADE singularities.
In case $n=5$ there is (up to permutation) precisely one more Delsarte surface with only ADE singularities, namely the surface with $\rho=25$ given by the equation:
\[Y^2X^3+Z^2Y^3+X^2Z^3+U^5.\]
\end{remark}

\section{Computation of the Picard Number: An example.}
Here we illustrate how we computed the Picard numbers in the table.
We do this for one example only. The same ideas work for the other cases.
We compute the Picard number of the surface $S$ given by the equation
\begin{equation}
 \label{formule}
X^{n}+Y^{n}+Z^{n-1}U+XYU^{n-2}=0.
\end{equation}

This corresponds to case 26 in the table of Appendix \ref{tabel}.

We determine the exponent matrix:
\[A=
\left(
\begin{array}{cccc}
n  &  0  &  0 & 0\\
0  &  n  &  0 & 0\\
0  &  0  &  n-1& 1\\
1  &  1  &  0 & n-2\\
\end{array}
\right).
\]

From this we find the three vectors in $V$ that generate $L$.
\[\v=\left(\frac{n-1}{n(n-2)},\frac{1}{n(n-2)},0,\frac{-1}{n-2}\right),\]
\[\w=\left(\frac{1}{n(n-2)},\frac{n-1}{n(n-2)},0,\frac{-1}{n-2}\right),\]
\[\u=\left(\frac{1}{(n-1)(n-2)},\frac{1}{(n-1)(n-2)},\frac{1}{n-1},\frac{-n}{(n-1)(n-2)}\right).\]

We use the following formula to compute $\#\Lambda$:
\[\# \Lambda =\# L-\# (L_0\cup D)-\#I-\# R.\]

We first determine the sets $N_i=\{v=(v_1,v_2,v_3,v_4) \in L: v_i=0\}$, for $i=1,2,3,4$.
\[N_1=\{i(n(n-2))\v+j((1-n)\v+\w)+k((n-n^2)\v+(n-1)\u):i,j,k\in \Z \},\]
\[N_2=\{i(n(n-2))\v+j((1-n)\v+\w)+k(-n\v+(n-1)\u):i,j,k\in \Z \},\]
\[N_3=\{i\v+j\w+k(n-1)\u:i,j,k\in \Z \},\]
\[N_4=\{i(n-2)\v+j(-\v+\w)+k(-n\v+(n-1)\u):i,j,k\in \Z \}.\]

We compute the intersection of these sets:
\[\bigcap_{i=1}^4 N_i=\{i(n(n-2))\v+j((1-n)\v+\w)+k((n-n^2)\v+(n-1)\u):i,j,k\in \Z \}.\]
From this we see that $L$ can be described by
\[L=\{i\v+j\w+k\u : 0\leq i<n(n-2), j=0, 0\leq k <n-1 \}.\]
This implies $\# L=n(n-1)(n-2)$ since this desciption gives a bijection between the set $L$ and the set
\[\{(i,j,k)\in \Z^3 : 0\leq i<n(n-2), j=0, 0\leq k <n-1 \}.\]

We compute the number of elements of $L_0$ using
\[\# L_0 = \sum_{i=1}^4 (-1)^{i+1}\sum_{1\leq d_1 <\ldots<d_i\leq 4} \# L\cap \bigcap N_{d_i}\]
This gives 
$\# L_0=n(n-2)$. 

Let $N_5$, $N_6$ and $N_7$ consist of all elements of the form $(a,-a,b,-b)$, $(a,b,-a,-b)$, $(a,b,-b,-a)$.
These are the elements of $D$.
These sets are given by:
\[N_5=\begin{array}{ll} 
\{i(n-2)\v+j(\v+\w)+k(-\v+\frac{n-1}{2}\u):i,j,k\in \Z \} & \mbox{if } n \equiv 1 \mod 2\\
\{i(n-2)\v+j(\v+\w)+k(-2\v+{n-1}\u):i,j,k\in \Z \}& \mbox{if } n \equiv 0 \mod 2
\end{array}
\]
\[N_6=\{i(n(n-2))\v+j((1-n)\v+\w)+k(-n(n-1)\v+\u):i,j,k\in \Z \}. \]
\[N_7=\{i(n(n-2))\v+j((1-n)\v+\w)+k(-n\v+\u):i,j,k\in \Z \}. \]
Note that for $N_5$ this result will depend on $n \mod 2$.
This behaviour is quite standard and happens for more cases.

We use this to compute 
\[\# (L_0\cup D)=
 \sum_{i=1}^7 (-1)^{i+1}\sum_{1\leq d_1 <\ldots<d_i\leq 7} \# L\cap \bigcap N_{d_i}=\begin{cases}
									    n^2-3 \mbox{ if }	n\equiv 1(2)\\
	                                                                    n^2-n-2 \mbox{ if }	n\equiv 0(2)\\
                                                                          \end{cases}
\]

We now compute the number of regular decomposable elements.
To do this we consider vectors of the form $r=(a,1/2,1/2+a,-2a)$.
We determine for which $a$ we have $r\in L$.
If $r$ belongs to $L$ then there exists a vector $(k,l,m)\in \Z^3$ such that
\[(k,l,m) \left(\begin{array}{cccc}
1  &  0  &  0 & -1\\
0  &  1  &  0 & -1\\
0  &  0  &  1 & -1\\
\end{array}
\right) A^{-1}=r.\]
Note that $r$ and $A$ are constructed in such a way that a rational solution for $(k,l,m)$ always exists. 
We only need to find out whether this solution is integral.

We calculate that
\[(k,l,m)=(a,1/2,1/2+a)\left(\begin{array}{cccc}
n-1  &  -1  &  0 \\
-1  &  n-1  &  0 \\
-1  &  -1  &  n-1 \\
\end{array}
\right).\]

The vector $(k,l,m)$ is only integral when $n\equiv 0(2)$ and $a=1/2$.
This corresponds to the solution $r=(1/2,1/2,0,0)$.
Since the third and fourth coordinates of $r$ are zero we find that $r\in L_0$ and hence $r\not\in R$.

We now consider the cases where $r$ is a permutation of $(a,1/2,1/2+a,-2a)$.
It turns out that we only find the result $r=(1/2,1/2,0,0)$, three more times, and no other integer solutions.
So none of the permutations of $(a,1/2,1/2+a,-2a)$ makes a contribution to the set $R$.

We move on to elements of $R$ that are permutations of $(a,1/2+a,1/2+2a,-4a)$.
We will give three examples in detail.

Consider the case $r=(a,1/2+a,1/2+2a,-4a)$.
With the same argument as before, we need to find an solution for
\[(k,l,m)=(a,1/2+a,1/2+2a)\left(\begin{array}{cccc}
n-1  &  -1  &  0 \\
-1  &  n-1  &  0 \\
-1  &  -1  &  n-1 \\
\end{array}
\right),\]
with $k$, $l$ and $m$ integers.
We find a solutions with $a=\frac{1}{4},\frac{3}{4}$ and $n\equiv0(4)$, furthermore we find a solution with
$a=\frac{1}{12},\frac{5}{12},\frac{7}{12},\frac{11}{12}$ and $n\equiv4(12)$.

The solutions $a=\frac{1}{4},\frac{3}{4}$ give respectively $r=(1/4,3/4,0,0)$ and $r=(3/4,1/4,0,0)$.
Since both solutions have a zero coordinate they are elements of $L_0$, and hence not of $R$.
The solutions $a=\frac{1}{12},\frac{5}{12},\frac{7}{12},\frac{11}{12}$ give respectively $r=(1/12,7/12,2/3,2/3)$,
$r=(5/12,11/12,1/3,1/3)$, $r=(7/12,1/12,2/3,2/3)$ and $r=(11/12,5/12,1/3,1/3)$.
These vectors are all elements of $R$.

Now we consider $r=(a,1/2+2a,-4a,1/2+a)$.
We find 
\[(k,l,m)=(a,1/2+a,1/2+2a)\left(\begin{array}{cccc}
n  &  0  &  0 \\
0  &  n-1  &  1 \\
1  &  0  &  n-2 \\
\end{array}
\right)\]
This equation has no integer solutions at all.

Let us now consider the vector $r=(a,1/2+2a,1/2+a,-4a)$
We need integral solutions for
\[(k,l,m)=(a,1/2+a,1/2+2a)\left(\begin{array}{cccc}
n-1  &  0    &  -1 \\
-1   &  n-1  &  -1 \\
-1   &  0    &  n-1 \\
\end{array}
\right)\]
We find that the only solution is given by $a=\frac{1}{2}$ with $n\equiv0(2)$.
This solutions corresponds to the vector $(1/2,0,1/2,0)$, and this is an element of $L_0$ and not of $R$.

Considering further permutations of $(a,1/2+a,1/2+2a,-4a)$, yield no new elements of $R$.

We will now look at the solutions of the form $r=(a,1/3+a,2/3+a,-3a)$.
As before $r$ is an element of $L$ precisely when
\[(k,l,m)=(a,1/3+a,2/3+a)\left(\begin{array}{cccc}
n-1  &  -1  &  0 \\
-1  &  n-1  &  0 \\
-1  &  -1  &  n-1 \\
\end{array}
\right)\]
has an integral solution for $k$, $l$ and $m$.
The only solutions occurs when $a=1/3$ and $n\equiv0(3)$ or when $a=5/6$ and $n\equiv3(6)$.
In the case $a=1/3$ we find that $r=(1/3,2/3,0,0)$, hence we find $r\in L_0$.
In the case $a=5/6$ we find $r=(5/6,1/6,1/2,1/2)$. This vector is decomposable, so we have $r\in D$.
Both cases yield no elements of $R$.

Finally considering permutations of $r=(a,1/3+a,2/3+a,-3a)$ yields no elements of $R$.

In total there are only four potential elements of $R$ namely:
\[\left(\frac{1}{12},\frac{7}{12},\frac{2}{3},\frac{2}{3}\right), \left(\frac{5}{12},\frac{11}{12},\frac{1}{3},\frac{1}{3}\right),\]
\[\left(\frac{7}{12},\frac{1}{12},\frac{2}{3},\frac{2}{3}\right), \left(\frac{11}{12},\frac{5}{12},\frac{1}{3},\frac{1}{3}\right).\]
These elements exist in $\Lambda$ precisely when $n\equiv 4(12)$.

Finally we will consider the finite set of exceptional solutions.
We have 22080 potential elements of the set $I$.
For each of these 22080 elements $r$ we can easily check whether it belongs to $I$.
We simply compute $rA$ and if the result consists of only integers then it belongs to $I$, otherwise not.
For example take $r=(\frac{1}{24},\frac{19}{24},\frac{1}{3},\frac{5}{6})$, one of the 22080 elements.
We find:
\[rA=\left(\frac{5}{6}+\frac{n}{24},\frac{5}{6}+\frac{19n}{24},\frac{n-1}{3},\frac{5n}{6}+\frac{2}{3}\right).\]
These coefficients are all integers iff $n\equiv 4(24)$.

Since we have to check $22080$ elements we resort to a computer search here.
The number of exeptional solutions is given by the formula:
\[\#I= 8\delta_{4,24}
+8\delta_{\{5,6,10\},30}
+16\delta_{6,40}
+12\delta_{9,42}
+16\delta_{10,48}
+16\delta_{\{5,6,10\},60}
+24\delta_{\{8,9\},84}
+32\delta_{22,120}
\]
Here \[\delta_{i,j}=
\begin{cases}
0 \mbox{ if } n\not\equiv  i(j)\\
1 \mbox{ if } n\equiv  i(j),
\end{cases},
\delta_{S,j}=\sum_{i\in S} \delta_{i,j}.\]

We now have enough information to compute the Lefschetz number using:
\[\# \Lambda =\# L-\# (L_0\cup D)-\#I-\# R\]
This gives:

\begin{eqnarray*}
\lambda &=&n^3-4n^2+2n+3+(n-1)\delta_{2,2}-4\delta_{4,12}-8\delta_{4,24}
-8\delta_{\{5,6,10\},30}
-16\delta_{6,40}
\\&&
-12\delta_{9,42}
-16\delta_{10,48}
-16\delta_{\{5,6,10\},60}
-24\delta_{\{8,9\},84}
-32\delta_{22,120}.
\end{eqnarray*}

For a surface of degree $n$ with only ADE-singularities we find after resolutions of singularities that $b_2=n^3-4n^2+6n-2$.
So the Picard number of the surface is given by (\ref{formule}):
\begin{eqnarray*}
\rho&=&
4n-5
-(n-1)\delta_{2,2}
+4\delta_{4,12}
+8\delta_{4,24}
+8\delta_{\{5,6,10\},30}
+16\delta_{6,40}
\\&&
+12\delta_{9,42}
+16\delta_{10,48}
+16\delta_{\{5,6,10\},60}
+24\delta_{\{8,9\},84}
+32\delta_{22,120}.
\end{eqnarray*}

\appendix
\section{Appendix}\label{tabel}
In this table we will give the formula's of all Delsarte surfaces with only ADE singularities.
The number on the left is simply an index. This is followed by the equation of the Delsarte surface and in the third column the singular points.
Below the second and third column the equation for the Picard number is given.

The case of the Fermat surface was already computed in \cite{Aoki}. 
Severak examples of smooth Delsarte surfaces where given by \cite{ShiodaPic}, and 
a systematic treatment of all smooth Delsarte surfaces was given in \cite{Jun}.
These cases have been marked in the table.

\[\begin{array}{|l|l|l|}
\hline
1&X^{n-2}YZ+Y^{n-2}ZU+Z^{n}+XU^{n-1}
& 
(0:1:0:0)\\
&&(1:0:0:0) \\
\cline{2-3}
&\multicolumn{2}{|l|}{ n^2-2n+2+2\delta_{3,3} }\\
\hline
2
& X^{n-2}YZ+Y^{n-2}ZU+Z^{n}+YU^{n-1} 

&
 (0:1:0:0) \\
&& (1:0:0:0) \\
\cline{2-3}
&\multicolumn{2}{|l|}{ 2n^2-3n-2 }\\
\hline
3&
  X^{n-2}YZ+XY^{n-2}Z+Z^{n}+XU^{n-1} 
 &
 (0:1:0:0) \\
&& (1:0:0:0) \\
\cline{2-3}
&\multicolumn{2}{|l|}{ 2n^2-5n+4+6\delta_{7,9}+6\delta_{13,18} }\\
\hline
4& X^{n-2}YZ+XY^{n-2}U+Z^{n}+XU^{n-1} &

 (0:1:0:0) \\
&& (1:0:0:0) \\
\cline{2-3}
&\multicolumn{2}{|l|}{ n^2-2+4\delta_{4,10}    }\\
\hline
5&
 X^{n-2}YZ+XY^{n-2}U+Z^{n}+YU^{n-1} & 
 
 (0:1:0:0) \\
&& (1:0:0:0) \\
\cline{2-3}
&\multicolumn{2}{|l|}{  n^2   }\\
\hline
6& X^{n-2}YZ+XY^{n-2}U+Z^{n}+ZU^{n-1} &

 (0:1:0:0) \\
&& (1:0:0:0) \\\cline{2-3}
&\multicolumn{2}{|l|}{  2n^2-5n+4    }\\
\hline

7& X^{n-2}YU+Y^{n-2}ZU+Z^{n}+XU^{n-1} 
& (0:1:0:0) \\
&& (1:0:0:0) \\\cline{2-3}
&\multicolumn{2}{|l|}{ n^2-n+\delta_{1,2}    }\\
\hline
8& X^{n-2}YU+Y^{n-2}ZU+Z^{n}+YU^{n-1} 
& (0:1:0:0) \\
&& (1:0:0:0) \\
&& (\sqrt[{n-2\;\;\;}]{-1}:0:0:1) \\
\cline{2-3}
&\multicolumn{2}{|l|}{  2n^2-3n-2   }\\
\hline
9& X^{n-2}YU+Y^{n-2}ZU+Z^{n}+ZU^{n-1} 
& (0:1:0:0) \\
&& (1:0:0:0) \\
&& (0:\sqrt[n-2\;\;\;]{-1}:0:1) \\\cline{2-3}
&\multicolumn{2}{|l|}{  2n^2-5n+4+(n-2)\delta_{2,2}+6\delta_{17,18}   }\\
\hline
10& X^{n-2}YU+XY^{n-2}U+Z^{n}+XU^{n-1} 
& (0:1:0:0) \\
&& (1:0:0:0) \\
&& (\sqrt[n-3\;\;\;]{-1}:1:0:0) \\
&& (0:\sqrt[n-2\;\;\;]{-1}:0:1) \\\cline{2-3}
&\multicolumn{2}{|l|}{  2n^2-5n+4+(n-2)\delta_{2,2}+4\delta_{5,5}+6\delta_{7,14}+6\delta_{6,18}+8\delta_{\{15,20\},30}   }\\
\hline
11& X^{n-2}YU+XY^{n-2}U+Z^{n}+ZU^{n-1} 
& (0:1:0:0) \\
&& (1:0:0:0) \\
&& (\sqrt[n-3\;\;\;]{-1}:1:0:0) \\\cline{2-3}
&\multicolumn{2}{|l|}{  2n^2-5n+4   }\\
\hline
12& X^{n-1}Y+Y^{n-1}Z+Z^{n}+XYU^{n-2} 
& (0:0:0:1) \\
&& (1:0:0:\sqrt[n-2\;\;\;]{-1}) \\\cline{2-3}
&\multicolumn{2}{|l|}{  2n^2-5n+4   }\\
\hline
13& X^{n-1}Y+Y^{n-1}Z+Z^{n}+XZU^{n-2} 
& (0:0:0:1) \\\cline{2-3}
&\multicolumn{2}{|l|}{  2n^2-5n+4   }\\
\hline
\end{array}
\]

\newpage
\[
\begin{array}{|l|l|l|}
\hline
14& X^{n-1}Y+Y^{n-1}Z+Z^{n}+YZU^{n-2} 
& (0:0:0:1) \\
&& (0:1:0:\sqrt[n-2\;\;\;]{-1}) \\\cline{2-3}
&\multicolumn{2}{|l|}{  2n^2-5n+4 +4\delta_{11,12}+6\delta_{16,18}   }\\
\hline
15
& X^{n-1}Y+Y^{n-1}U+Z^{n}+XYU^{n-2} 
& (0:0:0:1) \\
&& (1:0:0:\sqrt[n-2\;\;\;]{-1}) \\\cline{2-3}
&\multicolumn{2}{|l|}{  n^2-2n+2+n\delta_{2,2}+2\delta_{3,3}-2\delta_{2,4}
+8\delta_{5,15}+8\delta_{\{6,8\},24}+8\delta_{12,30} }\\

&\multicolumn{2}{|l|}{
 +12\delta_{14,42}
+16\delta_{\{12,20\},60}   }\\
\hline

16
& X^{n-1}Y+Y^{n-1}U+Z^{n}+XZU^{n-2} 
& (0:0:0:1) \\\cline{2-3}
&\multicolumn{2}{|l|}{  n+4\delta_{3,5}   }\\
\hline
17& X^{n-1}Y+Y^{n-1}U+Z^{n}+YZU^{n-2} 
& (0:0:0:1) \\\cline{2-3}
&\multicolumn{2}{|l|}{  n^2-2   }\\
\hline
18& X^{n-1}Y+XY^{n-1}+Z^{n}+XYU^{n-2} 
& (0:0:0:1) \\
&& (1:0:0:\sqrt[n-2\;\;\;]{-1}) \\
&& (0:1:0:\sqrt[n-2\;\;\;]{-1}) \\\cline{2-3}
&\multicolumn{2}{|l|}{  2n^2-5n+4
+(3n-10)\delta_{2,2}
+4\delta_{2,4}
+8\delta_{\{6,8\},12}
+16\delta_{12,15}

   }\\
&\multicolumn{2}{|l|}{  

+16\delta_{12,20}
+16\delta_{\{14,18,20\},24}
+16\delta_{\{12,20\},30}
+32\delta_{\{32,42,50\},60}   }\\
\hline
19& X^{n-1}Y+XY^{n-1}+Z^{n}+XZU^{n-2} 
& (0:0:0:1) \\\cline{2-3}
&\multicolumn{2}{|l|}{  n^2-2   }\\
\hline
20& X^{n-1}Z+Y^{n-1}Z+Z^{n}+XYU^{n-2} 
& (0:0:0:1) \\
&& (1:0:\sqrt[n-2\;\;\;]{-1}:0) \\\cline{2-3}
&\multicolumn{2}{|l|}{  3n^2-10n+8
+2\delta_{1,2}+n\delta_{2,2}
+4\delta_{5,12}
+12\delta_{8,14}

   }\\
&\multicolumn{2}{|l|}{  

+24\delta_{17,24}
+8\delta_{\{7,11\},30}+72\delta_{16,30}
+16\delta_{17,40}
+60\delta_{22,42}
+16\delta_{17,48}

   }\\
&\multicolumn{2}{|l|}{  

+48\delta_{\{37,41\},60}+16\delta_{46,60}
+40\delta_{23,66}
+48\delta_{40,78}
+24\delta_{\{22,37\},84}
+32\delta_{41,120}   }\\
\hline
21&
 X^{n-1}Z+Y^{n-1}U+Z^{n}+XYU^{n-2} 
& (0:0:0:1) \\\cline{2-3}
&\multicolumn{2}{|l|}{  3n-2+8\delta_{4,15} +8\delta_{9,20}+10\delta_{5,22}   }\\
\hline
22&
 X^{n-1}Z+Y^{n-1}U+Z^{n}+XZU^{n-2} 
& (0:0:0:1) \\
&& (1:0:0:\sqrt[n-2\;\;\;]{-1}) \\\cline{2-3}
&\multicolumn{2}{|l|}{  n^2-2n+2
+(n-1)\delta_{1,2}
+6\delta_{4,18}    }\\
\hline
23&
 X^{n-1}Z+Y^{n-1}U+Z^{n}+YZU^{n-2} 
& (0:0:0:1) \\\cline{2-3}
&\multicolumn{2}{|l|}{  n^2-2+4\delta_{9,10}   }\\
\hline
24&
 X^{n-1}U+Y^{n-1}U+Z^{n}+XYU^{n-2} 
& (0:0:0:1) \\
&& (1:\sqrt[n-1\;\;\;]{-1}:0:0) \\\cline{2-3}
&\multicolumn{2}{|l|}{  n^2-n
+\delta_{1,2}+n\delta_{2,2}
+(2n-2)\delta_{1,3}

   }\\
&\multicolumn{2}{|l|}{  

-2\delta_{4,6}
+8\delta_{5,20}
+8\delta_{6,30}
+12\delta_{9,36}

   }\\
&\multicolumn{2}{|l|}{  

+16\delta_{15,40}
+12\delta_{10,42}
+16\delta_{36,60}
+24\delta_{21,84}
   }\\
\hline

25&
 X^{n-1}U+Y^{n-1}U+Z^{n}+XZU^{n-2} 
& (0:0:0:1) \\
&& (1:0:\sqrt[n-1\;\;\;]{-1}:0) \\\cline{2-3}
&\multicolumn{2}{|l|}{  n^2-2   }\\
\hline

26&
 X^{n}+Y^{n}+Z^{n-1}U+XYU^{n-2} 
& (0:0:0:1) \\\cline{2-3}
&\multicolumn{2}{|l|}{  3n-5
+n\delta_{1,2}+\delta_{2,2}
+4\delta_{4,12}

   }\\
&\multicolumn{2}{|l|}{  

+8\delta_{4,24}
+8\delta_{\{5,6,10\},30}
+16\delta_{6,40}
+12\delta_{9,42}

   }\\
&\multicolumn{2}{|l|}{  

+16\delta_{10,48}
+16\delta_{\{5,6,10\},60}
+24\delta_{\{8,9\},84}
+32\delta_{22,120}
   }\\
\hline

27&
 X^{n}+Y^{n}+Z^{n-1}U+XZU^{n-2} 
& (0:0:0:1) \\
\cline{2-3}
&\multicolumn{2}{|l|}{ 3n-2+4\delta_{4,10}    }\\
\hline
28&
 X^{n}+Y^{n}+XZ^{n-1}+XYU^{n-2} 
& (0:0:0:1) \\\cline{2-3}
&\multicolumn{2}{|l|}{  n^2-3+(n-2)\delta_{1,2}+\delta_{2,2}+4\delta_{10,12}+8\delta_{\{22,26\},30}    }\\
\hline
\end{array}
\]

\[
\begin{array}{|l|l|l|}
\hline
29&
 X^{n}+Y^{n}+XZ^{n-1}+XZU^{n-2} 
& (0:0:0:1) \\
&& (0:0:\sqrt[n-2\;\;\;]{-1}:1) \\\cline{2-3}
&\multicolumn{2}{|l|}{  2n^2-5n+4+(n-2)\delta_{2,2}   }\\
\hline
30&
 X^{n}+Y^{n}+XZ^{n-1}+YZU^{n-2} 
& (0:0:0:1) \\\cline{2-3}
&\multicolumn{2}{|l|}{  n^2-2+4\delta_{5,12}+8\delta_{\{7,11\},30}   }\\
\hline
31&
 X^{n-1}Y+Y^{n-1}Z+XZ^{n-2}U+XYU^{n-2} 
& (0:0:0:1) \\
&& (0:0:1:0) \\
&& (1:0:0:\sqrt[n-2\;\;\;]{-1}) \\\cline{2-3}
&\multicolumn{2}{|l|}{  3n^2-10n+7
+n\delta_{1,2}+3\delta_{2,2}
+8\delta_{\{14,20\},30}   }\\
\hline
32&
 X^{n-1}Y+Y^{n-1}Z+XZ^{n-2}U+XZU^{n-2} 
& (0:0:0:1) \\
&& (0:0:1:0) \\
&& (0:0:1:\sqrt[n-3\;\;\;]{-1}) \\\cline{2-3}
&\multicolumn{2}{|l|}{  2n^2-5n+4   }\\
\hline
33&
 X^{n-1}Y+Y^{n-1}Z+XZ^{n-2}U+YZU^{n-2} 
& (0:0:0:1) \\
&& (0:0:1:0) \\
&& (0:1:0:\sqrt[n-2\;\;\;]{-1}) \\\cline{2-3}
&\multicolumn{2}{|l|}{  3n^2-10n+10   }\\
\hline
34&
 X^{n-1}Y+Y^{n-1}Z+YZ^{n-2}U+XZU^{n-2} 
& (0:0:0:1) \\
&& (0:0:1:0) \\\cline{2-3}
&\multicolumn{2}{|l|}{  2n^2-5n+6   }\\
\hline
35&  X^{n-1}Y+Y^{n-1}Z+XYZ^{n-2}+XZU^{n-2} 
& (0:0:0:1) \\
&& (0:0:1:0) \\
&& (1:0:\sqrt[n-2\;\;\;]{-1}:0) \\\cline{2-3}
&\multicolumn{2}{|l|}{  3n^2-10n+10+4\delta_{8,8}+8\delta_{20,24}   }\\
\hline
36& X^{n-1}Y+XY^{n-1}+XZ^{n-2}U+YZU^{n-2} 
& (0:0:0:1) \\
&& (0:0:1:0) \\\cline{2-3}
&\multicolumn{2}{|l|}{  4n^2-15n+16+4\delta_{10,12}+8\delta_{\{22,26\},30}   }\\
\hline
37& X^{n-1}Z+Y^{n-1}Z+XZ^{n-2}U+XYU^{n-2} 
& (0:0:0:1) \\
&& (0:0:1:0) \\
&& (1:\sqrt[n-1\;\;\;]{-1}:0:0) \\\cline{2-3}
&\multicolumn{2}{|l|}{  2n^2-5n+4   }\\
\hline

38& X^{n-1}Z+Y^{n-1}U+XZ^{n-2}U+XYU^{n-2} 
& (0:0:0:1) \\
&& (0:0:1:0) \\\cline{2-3}
&\multicolumn{2}{|l|}{  2n^2-5n+4    }\\
\hline

39&
 X^{n-1}Z+Y^{n-1}U+XZ^{n-2}U+XZU^{n-2} 
& (0:0:0:1) \\
&& (0:0:1:0) \\
&& (1:0:0:\sqrt[n-2\;\;\;]{-1}) \\
&& (0:0:1:\sqrt[n-3\;\;\;]{-1}) \\\cline{2-3}
&\multicolumn{2}{|l|}{  4n^2-15n+16+4\delta_{10,12}+8\delta_{\{22,26\},30}    }\\
\hline
40
& X^{n-1}Z+Y^{n-1}U+XZ^{n-2}U+YZU^{n-2} 
& (0:0:0:1) \\
&& (0:0:1:0) \\\cline{2-3}
&\multicolumn{2}{|l|}{  3n^2-7n+\delta_{1,2}   }\\
\hline
41& X^{n-1}Z+Y^{n-1}U+YZ^{n-2}U+XYU^{n-2} 
& (0:0:0:1) \\
&& (0:0:1:0) \\
&& (0:\sqrt[n-2\;\;\;]{-1}:1:0) \\\cline{2-3}
&\multicolumn{2}{|l|}{  3n^2-10n+10   }\\
\hline
\end{array}
\]
\[
\begin{array}{|l|l|l|}
\hline
42&
 X^{n-1}Z+Y^{n-1}U+YZ^{n-2}U+XZU^{n-2} 
& (0:0:0:1) \\
&& (0:0:1:0) \\
&& (1:0:0:\sqrt[n-2\;\;\;]{-1}) \\
&& (0:1:\sqrt[n-2\;\;\;]{-1}:0) \\\cline{2-3}
&\multicolumn{2}{|l|}{  4n^2-15n+16
+(n-4)\delta_{2,2}
+2\delta_{2,4}

   }\\
&\multicolumn{2}{|l|}{  

+8\delta_{8,12}
+12\delta_{\{8,14\},18}
+16\delta_{12,20}
+16\delta_{14,24}
+32\delta_{32,60}
   }\\
\hline
43&
 X^{n}+Y^{n}+XZ^{n-2}U+XYU^{n-2} 
& (0:0:0:1) \\
&& (0:0:1:0) \\\cline{2-3}
&\multicolumn{2}{|l|}{   n^2-2
+6\delta_{5,9}
+6\delta_{8,18}
  }\\
\hline

44&
 X^{n}+Y^{n}+XZ^{n-2}U+XZU^{n-2} 
& (0:0:0:1) \\
&& (0:0:1:0) \\
&& (0:0:1:\sqrt[n-3\;\;\;]{-1}) \\\cline{2-3}
&\multicolumn{2}{|l|}{  2n^2-5n+4+
2\delta_{4,6}
+16\delta_{9,24}+8\delta_{21,24}

   }\\
&\multicolumn{2}{|l|}{  

+8\delta_{6,30}
+32\delta_{21,60}
+32\delta_{55,120}
    }\\
\hline

45&
 X^{n}+Y^{n}+XZ^{n-2}U+YZU^{n-2} 
& (0:0:0:1) \\
&& (0:0:1:0) \\
\cline{2-3}
&\multicolumn{2}{|l|}{  n^2+n-8
+n\delta_{1,2}
+4\delta_{2,2}
+2\delta_{1,4}
+4\delta_{4,12}

   }\\
&\multicolumn{2}{|l|}{  

+8\delta_{7,24}
+8\delta_{\{6,10,13\},30}
+12\delta_{10,42}
+16\delta_{13,60}
   }\\
\hline
46&
 X^{n}+Y^{n}+Z^{n-1}U+XU^{n-1} &\mbox{First calculated in \cite{ShiodaPic}}\\
&&\mbox{Case \uppercase\expandafter{\romannumeral7\relax} in \cite{Jun}} \\
\cline{2-3}
&\multicolumn{2}{|l|}{  n
+4\delta_{3,12}+6\delta_{4,18}
 }\\
\hline
47&

 X^{n}+Y^{n}+Z^{n-1}U+ZU^{n-1} &\mbox{Case \uppercase\expandafter{\romannumeral3\relax} in \cite{Jun}} \\
\cline{2-3}
&\multicolumn{2}{|l|}{   
n^2-n+1
+(4n-9)\delta_{2,2}
+(4n-4)\delta_{2,3}
+8\delta_{4,4}

   }\\
&\multicolumn{2}{|l|}{  

-4\delta_{2,6}
-8\delta_{4,8}
+8\delta_{10,10}
-4\delta_{8,12}+12\delta_{12,12}
+32\delta_{\{6,12\},24}+16\delta_{8,24}

   }\\
&\multicolumn{2}{|l|}{  

+24\delta_{14,28}
+24\delta_{\{12,20\},30}+16\delta_{30,30}
+24\delta_{18,36}
+64\delta_{20,40}

   }\\
&\multicolumn{2}{|l|}{  

+48\delta_{14,42}
+64\delta_{24,48}
+144\delta_{\{12,20\},60}+192\delta_{30,60}
+40\delta_{44,66}
+48\delta_{54,72}

   }\\
&\multicolumn{2}{|l|}{  

+216\delta_{42,84}+72\delta_{56,84}
+64\delta_{\{30,72\},120}+128\delta_{60,120}
+96\delta_{78,156}
+96\delta_{72,180}
   }\\
\hline
48&
 X^{n}+Y^{n}+XZ^{n-1}+XU^{n-1} 
& (0:0:\sqrt[n-1\;\;\;]{-1}:1) \\
\cline{2-3}
&\multicolumn{2}{|l|}{  
2n^2-5n
+4\delta_{1,2}+2n\delta_{2,2}
-2\delta_{4,6}
+8\delta_{9,12}
+12\delta_{10,18}
+8\delta_{9,24}

+24\delta_{16,30}
   }\\
&\multicolumn{2}{|l|}{

+48\delta_{21,30}+32\delta_{25,30}
+48\delta_{22,42}
+32\delta_{\{21,25\},60}
+48\delta_{55,90}
+32\delta_{25,120}
   }\\
\hline
49&
 X^{n}+Y^{n}+XZ^{n-1}+YU^{n-1} &\mbox{Case \uppercase\expandafter{\romannumeral4\relax} in \cite{Jun}} 
\\
\cline{2-3}
&\multicolumn{2}{|l|}{  2n^2-5n+4
+4\delta_{4,12}
+24\delta_{6,30}+8\delta_{10,30}
 }\\
\hline
50&
 X^{n-2}YZ+Y^{n-2}ZU+XZ^{n-2}U+XU^{n-1} 
& (1:0:0:0) \\
&& (0:1:0:0) \\
&& (0:0:1:0) \\
&& (0:0:\sqrt[n-2\;\;\;]{-1}:1) \\\cline{2-3}
&\multicolumn{2}{|l|}{  4n^2-15n+16
+4\delta_{10,12}
+8\delta_{22,30}
    }\\
\hline
51&
 X^{n-2}YZ+Y^{n-2}ZU+XZ^{n-2}U+YU^{n-1} 
& (1:0:0:0) \\
&& (0:1:0:0) \\
&& (0:0:1:0) \\\cline{2-3}
&\multicolumn{2}{|l|}{  3n^2-10n+10    }\\
\hline
52&
 X^{n-2}YZ+Y^{n-2}ZU+YZ^{n-2}U+XU^{n-1} 
& (1:0:0:0) \\
&& (0:1:0:0) \\
&& (0:0:1:0) \\
&& (0:\sqrt[n-3\;\;\;]{-1}:1:0) \\\cline{2-3}
&\multicolumn{2}{|l|}{  3n^2-10n+10 
+(n-3)\delta_{1,2}
+8\delta_{\{9,15\},15}
+8\delta_{\{15,19\},20}
   }\\
\hline
\end{array}
\]
\[
\begin{array}{|l|l|l|}
\hline
53&
 X^{n-2}YZ+Y^{n-2}ZU+XYZ^{n-2}+XU^{n-1} 
& (1:0:0:0) \\
&& (0:1:0:0) \\
&& (0:0:1:0) \\
&& (\sqrt[n-3\;\;\;]{-1}:0:1:0) \\\cline{2-3}
&\multicolumn{2}{|l|}{  
4n^2-15n+16
+4\delta_{12,12}
+6\delta_{17,18}
   }\\
\hline

54&
 X^{n-2}YZ+Y^{n-2}ZU+XZ^{n-2}U+U^{n} 
& (1:0:0:0) \\
&& (0:1:0:0) \\
&& (0:0:1:0) \\\cline{2-3}
&\multicolumn{2}{|l|}{  
2n^2-5n+4
+4\delta_{5,5}
   }\\
\hline

55&
 X^{n-2}YZ+Y^{n-2}ZU+YZ^{n-2}U+U^{n} 
& (1:0:0:0) \\
&& (0:1:0:0) \\
&& (0:0:1:0) \\
&& (0:\sqrt[n-3\;\;\;]{-1}:1:0) \\\cline{2-3}
&\multicolumn{2}{|l|}{  
3n^2-8n
+4\delta_{1,2}+n\delta_{2,2}
+4\delta_{11,12}
+8\delta_{23,24}

   }\\
&\multicolumn{2}{|l|}{  

+8\delta_{\{23,27,28\},30}
+16\delta_{37,40}
+12\delta_{36,42}
+16\delta_{41,48}

   }\\
&\multicolumn{2}{|l|}{  

+16\delta_{\{53,57,58\},60}
+24\delta_{\{78,79\},84}
+32\delta_{101,120}
   }\\
\hline
56&
 X^{n-2}YZ+Y^{n-2}ZU+XYZ^{n-2}+U^{n} 
& (1:0:0:0) \\
&& (0:1:0:0) \\
&& (0:0:1:0) \\
&& (\sqrt[n-3\;\;\;]{-1}:0:1:0) \\\cline{2-3}
&\multicolumn{2}{|l|}{  
2n^2-3n-2
+4\delta_{9,10}
   }\\
\hline

57&
 X^{n-2}YZ+XY^{n-2}Z+XYZ^{n-2}+U^{n} 
& (1:0:0:0) \\
&& (0:1:0:0) \\
&& (0:0:1:0) \\
&& (0:\sqrt[n-3\;\;\;]{-1}:1:0) \\
&& (\sqrt[n-3\;\;\;]{-1}:0:1:0) \\
&& (\sqrt[n-3\;\;\;]{-1}:1:0:0) \\\cline{2-3}
&\multicolumn{2}{|l|}{  
3n^2-9n-4
+11\delta_{1,2}+3n\delta_{2,2}
+(2n-6)\delta_{3,3}
-6\delta_{6,6}
+12\delta_{10, 14}
+36\delta_{15, 18}

   }\\
&\multicolumn{2}{|l|}{  

+24\delta_{15, 20}
+24\delta_{24, 28}
+72\delta_{18, 30}
+36\delta_{27, 36}
+48\delta_{35, 40}

   }\\
&\multicolumn{2}{|l|}{  

+60\delta_{24, 42}
+48\delta_{48, 60}
+48\delta_{42, 78}
+72\delta_{63, 84}
 
  }\\
\hline
58&
 X^{n-1}Y+Y^{n-1}Z+Z^{n-1}U+XYU^{n-2} 
& (0:0:0:1) \\
&& (1:0:0:\sqrt[n-2\;\;\;]{-1}) \\\cline{2-3}
&\multicolumn{2}{|l|}{  
2n^2-5n+4
+4\delta_{6,12}
+8\delta_{12,30}
   }\\
\hline

59&
 X^{n-1}Y+Y^{n-1}Z+Z^{n-1}U+XZU^{n-2} 
& (0:0:0:1) \\\cline{2-3}
&\multicolumn{2}{|l|}{  n^2-2n+4    }\\
\hline
60&
 X^{n-1}Y+Y^{n-1}Z+Z^{n-1}U+YZU^{n-2} 
& (0:0:0:1) \\
&& (0:1:0:\sqrt[n-2\;\;\;]{-1}) \\\cline{2-3}
&\multicolumn{2}{|l|}{   3n^2-10n+10   }\\
\hline
61&
 X^{n-1}Y+Y^{n-1}Z+XZ^{n-1}+XYU^{n-2} 
& (0:0:0:1) \\
&& (1:0:0:\sqrt[n-2\;\;\;]{-1}) \\\cline{2-3}
&\multicolumn{2}{|l|}{   
2n^2-5n+4
+2\delta_{6,6}
   }\\
\hline
62&
 X^{n-1}Y+Y^{n-1}Z+YZ^{n-1}+XZU^{n-2} 
& (0:0:0:1) \\
&& (1:0:\sqrt[n-2\;\;\;]{-1}:0) \\\cline{2-3}
&\multicolumn{2}{|l|}{  3n^2-10n+10 }\\
\hline
63&
 X^{n-1}Y+Y^{n-1}U+Z^{n-1}U+XYU^{n-2} 
& (0:0:0:1) \\
&& (1:0:0:\sqrt[n-2\;\;\;]{-1}) \\
&& (0:1:\sqrt[n-1\;\;\;]{-1}:0) \\\cline{2-3}
&\multicolumn{2}{|l|}{  3n^2-10n+10   }\\
\hline
\end{array}
\]

\[
\begin{array}{|l|l|l|}
\hline
64&
 X^{n-1}Y+Y^{n-1}U+Z^{n-1}U+XZU^{n-2} 
& (0:0:0:1) \\
&& (0:1:\sqrt[n-1\;\;\;]{-1}:0) \\\cline{2-3}
&\multicolumn{2}{|l|}{  
2n^2-5n+4
+4\delta_{4,12}
+6\delta_{5,18}
   }\\
\hline
65&
 X^{n-1}Y+Y^{n-1}U+Z^{n-1}U+YZU^{n-2} 
& (0:0:0:1) \\
&& (0:1:\sqrt[n-1\;\;\;]{-1}:0) \\\cline{2-3}
&\multicolumn{2}{|l|}{  
2n^2-5n+4
   }\\
\hline
66
& X^{n-1}Y+Y^{n-1}U+YZ^{n-1}+XZU^{n-2} 
& (0:0:0:1) \\
&& (1:0:\sqrt[n-1\;\;\;]{-1}:0) \\\cline{2-3}
&\multicolumn{2}{|l|}{  
n^2-2n+2
+(n-1)\delta_{1,2}
+8\delta_{\{4,10\},15}
+8\delta_{\{5,9\},20}
   }\\
\hline
67&
 X^{n-1}Y+XY^{n-1}+Z^{n-1}U+XYU^{n-2} 
& (0:0:0:1) \\
&& (1:0:0:\sqrt[n-2\;\;\;]{-1}) \\
&& (0:1:0:\sqrt[n-2\;\;\;]{-1}) \\\cline{2-3}
&\multicolumn{2}{|l|}{  
4n^2-15n+13
+n\delta_{1,2}+3\delta_{2,2}
+4\delta_{10,12}
+12\delta_{9,14}
+24\delta_{10,24}

   }\\
&\multicolumn{2}{|l|}{  

+72\delta_{17,30}+8\delta_{\{22,26\},30}
+16\delta_{26,40}
+60\delta_{23,42}
+16\delta_{34,48}
+16\delta_{17,60}

   }\\
&\multicolumn{2}{|l|}{  
+48\delta_{\{22,26\},60}

+40\delta_{46,66}
+48\delta_{41,78}
+24\delta_{\{50,65\},84}
+32\delta_{82,120}
   }\\
\hline

68&
 X^{n-1}Y+XY^{n-1}+Z^{n-1}U+XZU^{n-2} 
& (0:0:0:1) \\\cline{2-3}
&\multicolumn{2}{|l|}{   
2n^2-3n-2+8\delta_{14,30}
   }\\
\hline
69
&
 X^{n-1}Y+Y^{n-1}Z+Z^{n-1}U+U^{n} &\mbox{Case \uppercase\expandafter{\romannumeral9\relax} in \cite{Jun}} \\\cline{2-3}
&\multicolumn{2}{|l|}{  n
+4\delta_{3,8}
+8\delta_{7,24}
    }\\
\hline

70
&
 X^{n-1}Y+Y^{n-1}Z+XZ^{n-1}+U^{n} &\mbox{First calculated in \cite{ShiodaPic}}\\
&&\mbox{Case \uppercase\expandafter{\romannumeral8\relax} in \cite{Jun}} \\\cline{2-3}
&\multicolumn{2}{|l|}{  
1
+\delta_{2,2}
+2n\delta_{3,3}
+6\delta_{\{4,6\},14}

   }\\
&\multicolumn{2}{|l|}{  

+12\delta_{\{4,20\},28}
+12\delta_{\{6,18\},42}
+24\delta_{\{18,24\},78}
   }\\
\hline
71&
 X^{n-1}Y+Y^{n-1}Z+YZ^{n-1}+U^{n} 
& (1:0:\sqrt[n-1\;\;\;]{-1}:0) \\\cline{2-3}
&\multicolumn{2}{|l|}{  
n^2-2n+2
+(n-1)\delta_{1,2}
+4\delta_{4,12}
+8\delta_{\{6,10\},30}
   }\\
\hline
72&
 X^{n-1}Y+Y^{n-1}U+Z^{n-1}U+U^{n} 
& (0:1:\sqrt[n-1\;\;\;]{-1}:0) \\\cline{2-3}
&\multicolumn{2}{|l|}{  
2n^2-5n+4
+(n-3)\delta_{1,2}
+2\delta_{1,4}
+8\delta_{7,12}

   }\\
&\multicolumn{2}{|l|}{  

+12\delta_{\{7,13\},18}
+16\delta_{11,20}
+16\delta_{13,24}
+32\delta_{31,60}
   }\\
\hline
73&
 X^{n-1}Y+Y^{n-1}U+YZ^{n-1}+U^{n} 
& (1:0:\sqrt[n-1\;\;\;]{-1}:0) \\\cline{2-3}
&\multicolumn{2}{|l|}{  2n^2-5n+4   }\\
\hline
74&
 X^{n-1}Y+XY^{n-1}+Z^{n-1}U+U^{n} 
&\mbox{Case \uppercase\expandafter{\romannumeral5\relax} in \cite{Jun}} \\
\cline{2-3}
&\multicolumn{2}{|l|}{  n^2-2n
+2\delta_{1,2}
+2\delta_{3,3}
+2\delta_{\{2,4\},6}

   }\\
&\multicolumn{2}{|l|}{  

+8\delta_{4,24}
+8\delta_{5,30}
+32\delta_{6,120}   }\\
\hline

75&
 X^{n}+Y^{n}+Z^{n}+XYU^{n-2} 
& (0:0:0:1) \\
\cline{2-3}
&\multicolumn{2}{|l|}{  
3n-2
+(3n-8)\delta_{2,2}
+4\delta_{4,4}
+8\delta_{\{6,8\},12}
+16\delta_{5,15}

   }\\
&\multicolumn{2}{|l|}{  

+16\delta_{10,20}
+16\delta_{\{6,8,12\},24}
+16\delta_{\{12,20\},30}
+32\delta_{\{12,20,32\},60}
   }\\
\hline

76&
 X^{n}+Y^{n}+Z^{n}+XU^{n-1} 
&\mbox{Case \uppercase\expandafter{\romannumeral2\relax} in \cite{Jun}} 
\\
\cline{2-3}
&\multicolumn{2}{|l|}{  
n^2-2n
+2n\delta_{1,2}+2\delta_{2,2}
-2\delta_{3,6}
+8\delta_{4,12}
+12\delta_{9,18}

   }\\
&\multicolumn{2}{|l|}{  

+8\delta_{16,24}
+32\delta_{6,30}+48\delta_{10,30}+24\delta_{15,30}
+48\delta_{21,42}
+32\delta_{\{36,40\},60}

   }\\
&\multicolumn{2}{|l|}{  

+48\delta_{36,90}
+32\delta_{96,120}   }\\
\hline

77&
 X^{n-2}YZ+Y^{n-2}ZU+XZ^{n-2}U+XYU^{n-2} 
& (1:0:0:0) \\
&& (0:1:0:0) \\
&& (0:0:1:0) \\
&& (0:0:0:1) \\\cline{2-3}
&\multicolumn{2}{|l|}{  
5n^2-21n+24
+\delta_{1,2}
+8\delta_{\{16,20\},20}
+8\delta_{\{20,26\},30}
+16\delta_{\{20,26,50,56\},60}
+32\delta_{\{86,110\},120}
   }\\
\hline
78&
 X^{n-1}Y+Y^{n-1}Z+Z^{n-1}U+XU^{n-1} &\mbox{First calculated in \cite{ShiodaPic}}\\
&&\mbox{Case \uppercase\expandafter{\romannumeral10\relax} in \cite{Jun}} \\
\cline{2-3}
&\multicolumn{2}{|l|}{   
n^2-n
+\delta_{1,2}
+8\delta_{\{4,8\},20}
+8\delta_{\{8,14\},30}
+16\delta_{\{8,14,38,44\},60}
+32\delta_{\{14,38\},120}
  }\\
\hline
79&
 X^{n-1}Y+Y^{n-1}Z+Z^{n-1}U+YU^{n-1} 
& (\sqrt[n-1\;\;\;]{-1}:0:0:1) \\\cline{2-3}
&\multicolumn{2}{|l|}{   n^2-2n+2
+2\delta_{3,6}  }\\
\hline
80&
 X^{n-1}Y+Y^{n-1}Z+Z^{n-1}U+ZU^{n-1} 
& (0:\sqrt[n-1\;\;\;]{-1}:0:1) \\\cline{2-3}
&\multicolumn{2}{|l|}{   2n^2-5n+4  }\\
\hline
\end{array}
\]
\[
\begin{array}{|l|l|l|}
\hline
81&
 X^{n-1}Y+Y^{n-1}Z+YZ^{n-1}+ZU^{n-1} 
& (1:0:\sqrt[n-1\;\;\;]{-1}:0) \\
&& (0:\sqrt[n-1\;\;\;]{-1}:0:1) \\\cline{2-3}
&\multicolumn{2}{|l|}{  
4n^2-15n+16
+4\delta_{10,12}
+8\delta_{22,30}
+24\delta_{26,30}   }\\
\hline
82&
 X^{n-1}Y+XY^{n-1}+Z^{n-1}U+ZU^{n-1} &\mbox{Case \uppercase\expandafter{\romannumeral6\relax} in \cite{Jun}} 
\\
\cline{2-3}
&\multicolumn{2}{|l|}{  
3n^2-9n+7
+\delta_{2,2}
+8\delta_{\{6,8\},12}
+16\delta_{12,20}
+16\delta_{\{6,8,14\},24}
+32\delta_{12,30}

   }\\
&\multicolumn{2}{|l|}{  
+16\delta_{20,30}
+48\delta_{30,42}
+96\delta_{12,60}+32\delta_{\{20,32,42,50\},60}
+48\delta_{14,84}
+64\delta_{72,120}
   }\\
\hline
83&

 X^{n}+Y^{n}+Z^{n}+U^{n} & \mbox{The Fermat surface described in  \cite{Aoki} and \cite{ShiodaPic}}\\
&& \mbox{Case \uppercase\expandafter{\romannumeral1\relax} in \cite{Jun}} \\
\cline{2-3}
&\multicolumn{2}{|l|}{  
3n^2-9n+7
+(24n-47)\delta_{2,2}
+(8n-24)\delta_{3,3}
-48\delta_{4, 4}
-96\delta_{6, 6}

   }\\
&\multicolumn{2}{|l|}{  

-48\delta_{8, 8}
-48\delta_{10, 10}
+144\delta_{12, 12}
+48\delta_{14, 14}
+192\delta_{15, 15}
 
   }\\
&\multicolumn{2}{|l|}{  

                                                    +432\delta_{18, 18}
                                                     +624\delta_{20, 20}
                                                     +288\delta_{21, 21}
                                                     +912\delta_{24, 24}
                                                     +240\delta_{28, 28}

   }\\
&\multicolumn{2}{|l|}{  

                                                     +2256\delta_{30, 30}
                                                     +432\delta_{36, 36}
                                                     +384\delta_{40, 40}
                                                     +3984\delta_{42, 42}
                                                     +384\delta_{48, 48}

   }\\
&\multicolumn{2}{|l|}{  

                                                     +4896\delta_{60, 60}
                                                     +720\delta_{66, 66}
                                                     +288\delta_{72, 72}
                                                     +768\delta_{78, 78}
                                                     +1584\delta_{84, 84}
 
   }\\
&\multicolumn{2}{|l|}{  

                                                    +576\delta_{90, 90}
                                                    +1728\delta_{120, 120}
                                                    +576\delta_{156, 156}
                                                    +576\delta_{180, 180}   }\\
\hline
\end{array}
\]

\bibliographystyle{plain}
\bibliography{ADEsing}

\end{document}